\UseRawInputEncoding
\documentclass[12pt]{article}

\usepackage{amsfonts,amsmath,amsthm}
\usepackage{amssymb}
\usepackage{algorithm}
\usepackage{algpseudocode}
\usepackage{xcolor}
\usepackage{graphicx}
\usepackage{stmaryrd}        % provides \llbracket and \rrbracket
\usepackage{multirow}
\usepackage{dsfont}

\usepackage[paper=a4paper,dvips,top=2cm,left=2cm,right=2cm,
foot=1cm,bottom=4cm]{geometry}

\newcommand{\vc}{{\bf c}}

\newcommand{\vx}{{\bf x}}
\newcommand{\vy}{{\bf y}}

\newcommand{\vz}{{\bf z}}

\newcommand{\R}{\mathbb{R}}

\newcommand{\rank}{\operatorname{rank}}

\newcommand{\A}{\mathcal{A}}

\newcommand{\1}{\mathbf{1}}

\newtheorem{Thm}{Theorem}[section]
\newtheorem{Def}[Thm]{Definition}

\newtheorem{Cor}[Thm]{Corollary}

\newtheorem{remark}[Thm]{Remark}

\begin{document}
	
\title{Sum of Squares Decompositions and Rank Bounds for Biquadratic Forms}
	\Large
	\author{Liqun Qi\footnote{%Jiangsu Provincial Scientific Research Center of Applied Mathematics, Nanjing 211189, China.
			Department of Applied Mathematics, The Hong Kong Polytechnic University, Hung Hom, Kowloon, Hong Kong.
			({\tt maqilq@polyu.edu.hk})}
		\and
		Chunfeng Cui\footnote{School of Mathematical Sciences, Beihang University, Beijing  100191, China.
			{\tt chunfengcui@buaa.edu.cn})}
		\and {and \
			Yi Xu\footnote{School of Mathematics, Southeast University, Nanjing  211189, China. Nanjing Center for Applied Mathematics, Nanjing 211135,  China. Jiangsu Provincial Scientific Research Center of Applied Mathematics, Nanjing 211189, China. ({\tt yi.xu1983@hotmail.com})}
		}
	}

	\date{\today}
	\maketitle
	
	\begin{abstract}
		We study positive semi-definite (PSD) biquadratic forms and their sum-of-squares (SOS) representations.
		For the class of partially symmetric biquadratic forms, we establish necessary and sufficient conditions for positive semi-definiteness and prove that every PSD partially symmetric biquadratic form is a sum of squares of bilinear forms.
		This extends the known result for fully symmetric biquadratic forms.   
		We describe an efficient computational procedure for constructing SOS decompositions, exploiting the Kronecker-product structure of the associated matrix representation.    {We introduce simple biquadratic forms.  For $m \ge 2$, 
		we present a $m \times 2$ PSD biquadratic form and show that it can be expressed as the sum of $m+1$ squares, but cannot be expressed as the sum of $m$ squares.}  This {provides a lower bound for sos rank of $m \times 2$ biquadratic forms, and} shows that previously proved results that a $2 \times 2$ PSD biquadratic form can be expressed as the sum of three squares, and a $3 \times 2$ PSD biquadratic form can be expressed as the sum of four squares, are tight.   {We also present an $3 \times 3$ SOS biquadratic form, which can be expressed as the sum of six squares, but not the sum  of five squares.}
       Moreover, we establish a universal upper bound $\operatorname{SOS\text{-}rank}(P) \le mn-1$ for any SOS biquadratic form, which improves the trivial bound $mn$.
		
		\medskip
		
		
	\end{abstract}

	\section{Introduction}
	
	Denote $\{1, \ldots, m\}$ as $[m]$.   Let $\A = (a_{ijkl})$, where $a_{ijkl} \in \R$ for $i, k \in [m], j, l \in [n]$. Then $\A$ is called an $m \times n$ {\bf biquadratic tensor}. Biquadratic tensors arise in solid mechanics \cite{KS76, ZR16}, statistics \cite{CHHS25},   spectral graph theory \cite{YYLDZY16}, and polynomial theory \cite{CQX25}.   If
	\begin{equation} \label{e1}
		a_{ijkl} = a_{kjil} = a_{klij}
	\end{equation}
	for $i, k \in [m], j, l \in [n]$, then $\A$ is called a {\bf symmetric biquadratic tensor}.
	Denote the set of all $m \times n$ biquadratic tensors by $BQ(m, n)$, and the set of all $m \times n$ symmetric biquadratic tensors by $SBQ(m, n)$.
	
	With $\A = (a_{ijkl}) \in BQ(m, n)$, we may define a homogeneous polynomial $P$ by
	\[
	P(\mathbf{x}, \mathbf{y}) = \sum_{i,k=1}^{m} \sum_{j,l=1}^{n} a_{ijkl} x_i x_k y_j y_l.
	\]
	We call $P$ a {\bf biquadratic form}.  A PSD (positive semi-definite) biquadratic form is one for which $P(\mathbf{x}, \mathbf{y}) \geq 0$ for all $\mathbf{x}, \mathbf{y}$. It is SOS (sum of squares) if it can be written as a finite sum of squares of bilinear forms.   If the biquadratic form $P$ is PSD or SOS respectively, then the biquadratic tensor $\A$ is also said to be PSD or SOS respectively.    Note that while a biquadratic tensor uniquely defines a biquadratic form, a biquadratic form corresponds to infinitely many biquadratic tensors, but uniquely corresponds to one symmetric biquadratic tensor.
	
	A PSD polynomial is nonnegative everywhere.   A very important problem in algebra and optimization is to identify structured PSD tensors and forms.   As Hilbert \cite{Hi88} proved in 1888, a PSD form may not be SOS in general.  On the other hand,
	the existence of an SOS decomposition makes the problem of verifying nonnegativity computationally tractable via semidefinite programming.    Thus, it is also very important to identify whether a given structured PSD tensor or form class is SOS or not.
	
	\bigskip
	
	Biquadratic forms arise naturally in several areas of applied mathematics and engineering \cite{QC25}.
	In elasticity theory, the strong ellipticity condition for isotropic materials is governed by a biquadratic form whose positivity guarantees wave propagation.
	In statistics and machine learning, biquadratic kernels appear in higher-order learning models and tensor-based data analysis.
	Moreover, verifying nonnegativity of biquadratic forms is a fundamental subproblem in polynomial optimization and semidefinite programming relaxations \cite{BPT12, Par03}.
	Thus, understanding when a positive semidefinite biquadratic form admits an SOS decomposition not only advances classical real algebraic geometry, but also provides practical tools for certifying positivity in these applications.
	
	Very recently, several structured PSD biquadratic tensor classes are identified.   In \cite{QC25}, three PSD biquadratic tensor classes, diagonally dominated symmetric biquadratic tensors, symmetric M-biquadratic tensors and  symmetric $\mathrm{B}_{0}$-biquadratic tensors were identified.   In \cite{QCCX25}, completely positive biquadratic tensors were introduced.    All completely positive biquadratic tensors are SOS.   Two subclasses of completely positive biquadratic tensors, biquadratic Cauchy tensors and biquadratic Pascal tensors, were identified there.   In \cite{XCQ25}, it was shown that
	diagonally dominated symmetric biquadratic tensors are SOS tensors.
	
	A PSD biquadratic form may not be SOS.  In 1975, Choi \cite{Ch75} presented a $3 \times 3$ PSD biquadratic form which is not SOS.  In \cite{XCQ25}, it was shown that such a problem is caused by non-symmetry.  {\bf Symmetric biquadratic forms} were introduced in \cite{XCQ25}, and it was shown that all PSD symmetric biquadratic forms are SOS, no matter how big $m$ and $n$ are.   A biquadratic form is said to be symmetric if it remains invariant under permutations of  the indices in $x$-variables and $y$-variables respectively.
	
	In this paper, we consider the intermediate class of partially symmetric biquadratic forms, which are symmetric only either in the $\vx$ variables or the $\vy$ variables.
	We establish necessary and sufficient conditions for positive semi-definiteness of monic  partially symmetric biquadratic forms and, more importantly, prove that every PSD partially symmetric biquadratic form is SOS.
	This resolves an open question about the SOS property under partial symmetry and extends the known result for fully symmetric forms.
	
	Furthermore, we describe an efficient computational procedure for constructing SOS decompositions, exploiting the Kronecker-product structure of the associated matrix representation.
	
	{We introduce simple biquadratic forms.  For $m \ge 2$,
		we present a $m \times 2$ PSD biquadratic form and show that it can be expressed as the sum of $m+1$ squares, but cannot be expressed as the sum of $m$ squares.  This provides a lower bound for sos rank of $m \times 2$ biquadratic forms, and shows that previously proved results that a $2 \times 2$ PSD biquadratic form can be expressed as the sum of three squares, and a $3 \times 2$ PSD biquadratic form can be expressed as the sum of four squares, are tight.  We also present an $3 \times 3$ SOS biquadratic form, which can be expressed as the sum of six squares, but not the sum  of five squares.}

   In addition, we derive a universal upper bound $\operatorname{SOS\text{-}rank}(P) \le mn-1$ for any SOS biquadratic form, which improves the {trivial bound $mn$}.
	
	\section{Monic Partially Symmetric Biquadratic Forms}
	\label{sec:x_symmetric_forms}

		Let $\A = (a_{ijkl}) \in SBQ(m , n)$, $\vx \in \R^m$ and $\vy \in \R^n$. Then $\A \cdot \vy \vx \vy \in \R^m$ and $(\A \cdot \vy \vx \vy)_i = \sum_{k=1}^m \sum_{j,l=1}^n a_{ijkl} y_j x_k y_l$ for $i \in [m]$, and $\A \vx \cdot \vx \vy \in \R^n$ with $(\A \vx \cdot \vx \vy)_j = \sum_{i,k=1}^m \sum_{l=1}^n a_{ijkl} x_i x_k y_l$ for $j \in [n]$. If there are $\lambda \in \R$, $\vx \in \R^m$, $\|\vx\|_2 = 1$, $\vy \in \R^n$, $\|\vy\|_2 = 1$, such that
	\[
	\A \cdot \vy \vx \vy = \lambda \vx, \quad \A \vx \cdot \vx \vy = \lambda \vy,
	\]
	then $\lambda$ is called an M-eigenvalue of $\A$, with $\vx$ and $\vy$ as its M-eigenvectors \cite{QDH09}.	For other works on biquadratic forms and biquadratic tensors, see \cite{Ca73, CQX25, DLQY20, WSL20, Zh23}

	Suppose that $\A = (a_{ijkl})$, where $a_{ijkl} \in \R$ for $i, k \in [m], j, l \in [n]$, is an $m \times n$ biquadratic tensor. Let
	\begin{equation} \label{bqform}
		P(\vx, \vy) = \sum_{i,k=1}^m \sum_{j,l=1}^n a_{ijkl}x_i y_j x_k y_l,
	\end{equation}
	where $\vx \in \R^m$ and $\vy \in \R^n$. Then $P$ is called a \textbf{biquadratic form}. If $P(\vx, \vy) \ge 0$ for all $\vx \in \R^m$ and $\vy \in \R^n$, then $P$ is called \textbf{positive semi-definite (PSD)}. If
	\[
	P(\vx, \vy) = \sum_{p=1}^r f_p(\vx, \vy)^2,
	\]
	where $f_p$ for $p \in [r]$ are bilinear forms, then we say that $P$ is \textbf{sum-of-squares (SOS)}. The smallest $r$ is called the \textbf{SOS rank} of $P$. Clearly, $P$ is PSD or SOS if and only if $\A$ is PSD or SOS, respectively.
	
	While a biquadratic form $P$ may be constructed from a biquadratic tensor $\A$, the tensor $\A$ is not unique to $P$. However, there is a unique \emph{symmetric} biquadratic tensor $\A$ associated with $P$.

    The following definition extends the definition of symmetric forms \cite{GKR17}.

	\begin{Def}
		Suppose $P(\vx, \vy) = P(x_1, \ldots, x_m, y_1, \ldots, y_n)$. If
		\[
		P(x_1, \ldots, x_m, y_1, \ldots, y_n) = P(x_{\sigma(1)}, \ldots, x_{\sigma(m)}, y_1, \ldots, y_n)
		\]
		for any permutation $\sigma$, then $P$ is called \textbf{$x$-symmetric}. Similarly, we define \textbf{$y$-symmetric biquadratic forms}. If $P$ is both $x$-symmetric and $y$-symmetric, we call $P$ a \textbf{symmetric biquadratic form}.  If $P$ is either $x$-symmetric or $y$-symmetric, we call $P$ a \textbf{partially symmetric biquadratic form}.
	\end{Def}
	
	Since it was proved in \cite{XCQ25} that all PSD symmetric biquadratic forms are SOS, the question is now: whether a given PSD $x$-symmetric biquadratic form is SOS or not.
	
	Note that an $x$-symmetric $m \times n$ biquadratic form $P$ has  $2n^2$ free coefficients. If we fix the diagonal coefficients as $1$, we may write such a form as
	\begin{align}
		\nonumber &P(\vx, \vy)\\
		= &\sum_{i=1}^m \sum_{j=1}^n x_i^2 y_j^2
		+ \sum_{i \neq k} \sum_{j, l=1}^n a_{jl} x_i x_k y_jy_l  + \sum_{i=1}^m \sum_{j \neq l} b_{jl}x_i^2 y_j y_l \label{equ:monic_sbq}\\
		\nonumber =&{(\vx^\top\vx)(\vy^\top\vy) + \left((\1_m^\top\vx)^2 - \vx^\top\vx \right)\left(\vy^\top A \vy\right) + \left(\vx^\top\vx\right)\left(\vy^\top B \vy\right),}
	\end{align}
	where $A=(a_{jl})$ and $B=(b_{jl})$ are $n\times n$ symmetric matrices, and $B$ has zero diagonal (i.e., $b_{jj}=0$ for all $j$). We call such a form a \textbf{monic $x$-symmetric biquadratic form}.

	\begin{remark}
		{Similarly, we can define monic $y$-symmetric $m \times n$ biquadratic form  as follows
			\begin{align*}
				&P(\vx, \vy)\\
				= &\sum_{i=1}^m \sum_{j=1}^n x_i^2 y_j^2
				+ \sum_{i, k=1}^m\sum_{j \neq l}  a_{ik} x_i x_k y_jy_l  +  \sum_{i \neq k}\sum_{j=1}^n b_{ik}x_ix_k y_j^2 \\
				=&(\vx^\top\vx)(\vy^\top\vy) + \left(\vx^\top A \vx\right)\left((\1_n^\top\vy)^2 - \vy^\top\vy \right) + \left(\vx^\top B\vx\right)\left(\vy^\top  \vy\right), % <- corrected exponents
			\end{align*}
			where $A=(a_{ik})$ and $B=(b_{ik})$ are $m\times m$ symmetric matrices, and $B$ has zero diagonal (i.e., $b_{i}=0$ for all $i$).	Our results on $x$-symmetric in this paper may be naturally extended to $y$-symmetric biquadratic forms.}
	\end{remark}

	\subsection{Positive Semi-definiteness of Monic $x$-Symmetric Forms}
	
	\begin{Thm}\label{thm:psd_xsym}
		Let $P$ be an $m \times n$ monic $x$-symmetric biquadratic form as in \eqref{equ:monic_sbq}, and let $\A\in SBQ(m,n)$ be the corresponding symmetric biquadratic tensor.
		Let $A=(a_{jl})$ and $B=(b_{jl})$ be the $n\times n$ symmetric matrices of coefficients.
		Then $P$ is PSD if and only if the following two matrix inequalities hold:
		\begin{equation}\label{Thm2.2_inequlaity}
			I + B - A \succeq 0, \qquad I + B + (m-1)A \succeq 0,
		\end{equation}
		where $I$ denotes the $n \times n$ identity matrix.
	\end{Thm}
	
	\begin{proof}
		$P$ is positive semi-definite if and only if all M-eigenvalues are nonnegative.
		A constant $\lambda$ is an M-eigenvalue of $\A$ if there exists  unit vectors   $\vx$ and $\vy$  such that
		\begin{subequations}\label{eq:Meig_xsym}
			\begin{align}
				\bigl[\bigl(\vy^\top(I+B-A)\vy\bigr)I_m
				+ \bigl(\vy^\top A \vy\bigr)\mathbf{1}_m\mathbf{1}_m^\top\bigr]\vx & =  \lambda\vx, \label{eq:sub1} \\
				\bigl[(I+B-A)+(\vx^\top\mathbf{1}_m)^2A\bigr]\vy &= \lambda\vy.\label{eq:sub2}
			\end{align}
		\end{subequations}
		{Here, the second equality follows from $\vx^\top\vx=1$.}
		
		Equation \eqref{eq:sub1} inspires us to consider the following two cases with respect to the eigenvectors $\vx$ and $\vy$: % <- corrected
		
		Case (i) $\vx^T\mathbf{1}_m=0$ or  $A\vy=0$. In this case, \eqref{eq:Meig_xsym} reduces to
		\[	\bigl(\vy^\top(I+B-A)\vy\bigr)\vx  =  \lambda_1\vx   \text{ and }
		(I+B-A)\vy= \lambda_1\vy.\]
		Consequently, 		$\lambda_1$ is an  M-eigenvalue of $\A$ if  and only if
		\[(I+B-A)\vy = \lambda_1\vy \text{ and } \lambda_1=\vy^\top(I+B-A)\vy.\]
		Thus, all such M-eigenvalues $\lambda_1$ are nonnegative if and only if $I+B-A\ge 0$.

		Case (ii) $\vx^T\mathbf{1}_m\neq0$ and   $A\vy\neq 0$.
		Equation \eqref{eq:sub1} implies that $\vx$ is parallel to $\mathbf{1}_m$. Together with the condition that $\vx$ is a unit vector, this forces $\vx = \frac{1}{\sqrt{m}}\mathbf{1}_m$. Substituting this expression for $\vx$ into \eqref{eq:Meig_xsym} then yields
		\[(I+B+(m-1)A)\vy = \lambda\vy \text{ and } \lambda_2=\vy^\top(I+B+(m-1)A)\vy.\]
		Consequently, all  M-eigenvalues $\lambda_2$ are  nonnegative if and only if $\vy^\top(I+B+(m-1)A)\vy\ge 0$ for any {$A \vy\neq  0.$}
		
		Suppose that $P$ is PSD, then it follows from Case (i) that $I+B-A\ge 0$. Now, let $\mathbf{y}_i$ be an orthonormal eigenbasis of $I + B + (m-1)A$. For each eigenvector $\mathbf{y}_i$, we distinguish two cases:
		If $A\vy_i\neq 0$, Case (ii) directly gives  $\vy_i^\top(I+B+(m-1)A)\vy_i\ge 0$.
		Otherwise, if $A\vy_i= 0$, then  $$\vy_i^\top(I+B+(m-1)A)\vy_i = \vy_i^\top(I+B-A)\vy_i\ge 0.$$
		where the inequality follows from $I+B-A\ge 0$.  Since $\{\vy_i\}$ is an orthonormal basis of $\mathbb R^n$, it follows that    $I+B+(m-1)A\ge 0$.
		
		On the other hand, suppose that \eqref{Thm2.2_inequlaity} holds, the M-eigenvalues in Cases (i) and (ii) are  both  nonnegative, and hence $P$ is PSD.  This completes the proof.
	\end{proof}

	\begin{remark}
		When $P$ is also $y$-symmetric, the matrices $A$ and $B$ reduce to
		\[
		A=(a-c)I_n+c\mathbf{1}_n\mathbf{1}_n^\top,\qquad
		B=b(\mathbf{1}_n\mathbf{1}_n^\top-I_n),
		\]
		{for three constants $a,b,c\in\mathbb R$. Then,}  the two matrix inequalities in \eqref{Thm2.2_inequlaity} reproduce the four linear inequalities (i)--(iv) of Theorem~4.2 in \cite{XCQ25}.
		Thus Theorem~\ref{thm:psd_xsym} extends the earlier result to the more general $x$-symmetric setting.
	\end{remark}
	
	\subsection{SOS Property of Monic $x$-Symmetric Forms}
	
	We now prove that every PSD $x$-symmetric biquadratic form is SOS.
	This resolves the open question raised in the introduction.
	
	\begin{Thm}\label{thm:sos_xsym}
		Let $P$ be an $m \times n$ monic $x$-symmetric biquadratic form as in \eqref{equ:monic_sbq}.
		If $P$ is positive semi-definite, then $P$ is a sum of squares of bilinear forms.
	\end{Thm}
	
	\begin{proof}
		Let $A$ and $B$ be the $n \times n$ symmetric matrices as in Theorem~\ref{thm:psd_xsym}.
		Since $P$ is PSD, Theorem~\ref{thm:psd_xsym} yields
		\[
		Q := I_n + B - A \succeq 0, \qquad R := I_n + B + (m-1)A \succeq 0.
		\]
		Let $\A$ be the unique symmetric biquadratic tensor corresponding to $P$.
		Define the $mn \times mn$ matrix $M$ by
		\[
		M_{(i,j),(k,l)} = a_{ijkl}, \quad \text{for } i,k \in [m], \; j,l \in [n],
		\]
		where $a_{ijkl}$ are the entries of $\A$.
		Then $M$ can be written in block form as
		\[
		M = I_m \otimes (I_n + B - A) + (\mathbf{1}_m \mathbf{1}_m^\top) \otimes A
		= I_m \otimes Q + \frac{1}{m} (\mathbf{1}_m \mathbf{1}_m^\top) \otimes (R - Q).
		\]
		For any vectors $\vx \in \R^m$ and $\vy \in \R^n$, let $\vz = \vx \otimes \vy$. Then
		\[
		P(\vx, \vy) = \sum_{i,k=1}^m \sum_{j,l=1}^n a_{ijkl} x_i x_k y_j y_l = \vz^\top M \vz.
		\]
		We now show that $M$ is positive semidefinite.
		Let $\vz \in \R^{mn}$ be partitioned as $\vz = (\vz_1^\top, \dots, \vz_m^\top)^\top$ with $\vz_i \in \R^n$, and set $\mathbf{s} = \sum_{i=1}^m \vz_i$. Then,
		\begin{align*}
			\vz^\top M \vz &= \sum_{i=1}^m \vz_i^\top Q \vz_i + \mathbf{s}^\top A \mathbf{s} \\
			&= \sum_{i=1}^m \vz_i^\top Q \vz_i + \frac{1}{m} \mathbf{s}^\top (R - Q) \mathbf{s} \\
			&= \sum_{i=1}^m \vz_i^\top Q \vz_i - \frac{1}{m} \mathbf{s}^\top Q \mathbf{s} + \frac{1}{m} \mathbf{s}^\top R \mathbf{s}.
		\end{align*}
		Observe that
		\[
		\sum_{i=1}^m \vz_i^\top Q \vz_i - \frac{1}{m} \mathbf{s}^\top Q \mathbf{s}
		= \sum_{i=1}^m \left( \vz_i - \frac{1}{m} \mathbf{s} \right)^\top Q \left( \vz_i - \frac{1}{m} \mathbf{s} \right).
		\]
		Hence,
		\[
		\vz^\top M \vz = \sum_{i=1}^m \left( \vz_i - \frac{1}{m} \mathbf{s} \right)^\top Q \left( \vz_i - \frac{1}{m} \mathbf{s} \right) + \frac{1}{m} \mathbf{s}^\top R \mathbf{s}.
		\]
		Since $Q \succeq 0$ and $R \succeq 0$, both terms are nonnegative for every $\vz$. Therefore $M \succeq 0$.
		
		Because $M$ is positive semidefinite, it can be decomposed as $M = \sum_{p=1}^{mn} \mathbf{w}_p \mathbf{w}_p^\top$ (for example, by its spectral decomposition). Then
		\[
		P(\vx, \vy) = \vz^\top M \vz = \sum_{p=1}^{mn} (\mathbf{w}_p^\top \vz)^2.
		\]
		Each inner product $\mathbf{w}_p^\top \vz$ is a bilinear form in $\vx$ and $\vy$, because $\vz = \vx \otimes \vy$ and $\mathbf{w}_p{\in \mathbb R^{mn}}$  corresponds to a matrix $W_p{\in \mathbb R^{m\times n}}$ such that $\mathbf{w}_p^\top \vz = \vx^\top W_p \vy$. Hence $P$ is a sum of squares of bilinear forms.
	\end{proof}
	
	\begin{remark}
		Theorem~\ref{thm:sos_xsym}, together with the result of \cite{XCQ25} for fully symmetric forms, shows that for two natural symmetry classes (fully symmetric and $x$-symmetric) every PSD biquadratic form is SOS.
		This provides a partial answer to the open problem of characterizing which structured PSD biquadratic forms are SOS.
	\end{remark}
	
	\subsection{Bounding the SOS Rank}
	
	For a PSD monic $x$-symmetric biquadratic form $P$, the constructive proof of Theorem~\ref{thm:sos_xsym} yields an explicit upper bound on the minimum number of squares required in an SOS decomposition.
	
	\begin{Thm}[SOS rank bound]\label{thm:sos_rank_bound}
		Let $P$ be an $m \times n$ monic $x$-symmetric biquadratic form as in \eqref{equ:monic_sbq} that is positive semi-definite, and let
		\[
		Q = I_n + B - A, \qquad R = I_n + B + (m-1)A,
		\]
		where $A,B$ are the symmetric coefficient matrices from \eqref{equ:monic_sbq}.
		Then $Q \succeq 0$ and $R \succeq 0$, and the SOS rank of $P$ satisfies
		\begin{equation} \label{e4}
			\operatorname{SOS\text{-}rank}(P) \; \le \; \operatorname{rank}(R) + (m-1)\cdot \operatorname{rank}(Q).
		\end{equation}
	\end{Thm}
	
	\begin{proof}
		Because $P$ is PSD, Theorem~\ref{thm:psd_xsym} guarantees $Q \succeq 0$ and $R \succeq 0$.
		Let $\A$ be the unique symmetric biquadratic tensor corresponding to $P$ and define the symmetric matrix
		\begin{equation} \label{ee5}
			M = I_m \otimes Q + \tfrac{1}{m}(\mathbf{1}_m\mathbf{1}_m^\top) \otimes (R-Q)
			\in \R^{mn\times mn},
		\end{equation}
		as in the proof of Theorem~\ref{thm:sos_xsym}.
		For any $\vx\in\R^m,\vy\in\R^n$ set $\vz=\vx\otimes\vy$; then
		\begin{equation} \label{ee6}
			P(\vx,\vy) = \vz^\top M\vz.
		\end{equation}
		The matrix $M$ is positive semidefinite and, via the orthogonal transformation
		$U$ whose first column is $\mathbf{1}_m/\sqrt{m}$ and remaining columns an orthonormal basis of $\mathbf{1}_m^\perp$, one obtains the block diagonal form
		\begin{equation} \label{ee7}
			(U^\top\otimes I_n)\, M\, (U\otimes I_n)
			= \operatorname{diag}\bigl(R,\, Q,\, \dots,\, Q\bigr),
		\end{equation}
		where $R$ appears once and $Q$ appears $m-1$ times.
		Consequently
		\begin{equation} \label{ee8}
			\operatorname{rank}(M)=\operatorname{rank}(R)+(m-1)\operatorname{rank}(Q).
		\end{equation}
		
		{Suppose that $M = \sum_{p=1}^{\operatorname{rank}(M)} \mathbf{w}_p\mathbf{w}_p^\top$. Then
			\[
			P(\vx,\vy)
			= \sum_{p=1}^{\operatorname{rank}(M)} \bigl(\mathbf{w}_p^\top(\vx\otimes\vy)\bigr)^2= \sum_{p=1}^{\operatorname{rank}(M)} \bigl(\vx^\top W_p \vy\bigr)^2.
			\]
			Here, $W_p\in\mathbb R^{m\times n}$ is the matricization of $w_p\in\mathbb R^{m n}$.
			This combined with \eqref{ee8} shows that  $P$ is a sum of $\operatorname{rank}(M)$ squares of bilinear forms, proving \eqref{e4}.}
	\end{proof}
	
	\begin{remark}\label{rem:relationship_bounds}
		\textbf{Relationship between Theorem~\ref{thm:sos_rank_bound} and Theorem~4 of \cite{CQX25}.}
		Theorem~\ref{thm:sos_rank_bound} provides a refined, structure-dependent bound for $x$-symmetric forms:
		$\operatorname{SOS\text{-}rank}(P) \le \operatorname{rank}(R) + (m-1)\operatorname{rank}(Q)$.
		For the completely squared form $P_{m,n}$ {with $A=B=O$,} we have  $\operatorname{rank}(Q)=\operatorname{rank}(R)=n$,
		so Theorem~\ref{thm:sos_rank_bound} yields the same $mn$ bound as Theorem~4 of \cite{CQX25}.
		However, for many $x$-symmetric forms, the matrices $Q$ and $R$ have rank much smaller than $n$,
		making Theorem~\ref{thm:sos_rank_bound} significantly sharper than the universal $mn$ bound.
		The value of Theorem~\ref{thm:sos_rank_bound} lies in its ability to exploit the $x$-symmetry
		to obtain tighter bounds and to guide the efficient computational procedure of Section~\ref{sec:computation}.
	\end{remark}

	Moreover, %equality in (\ref{e4}) is
	an SOS decomposition whose rank equals  the upper bound in  (\ref{e4}) is  attained in Section~\ref{sec:computation}.
	Specifically, the explicit spectral  construction of Section~\ref{sec:computation} produces exactly this many squares (one for each positive eigenvalue of $R$ and $m-1$ copies of each positive eigenvalue of $Q$), so the bound is attainable by that construction.

	\section{Computational Method for SOS Decomposition}
	\label{sec:computation}
	
	Theorem~\ref{thm:sos_xsym} provides a constructive proof that every PSD $x$-symmetric biquadratic form admits an SOS decomposition.  Here we outline a concrete numerical procedure to compute such a decomposition.

	{We may rewrite the biquadratic form in \eqref{equ:monic_sbq}  as follows, % <- corrected
		\begin{align*}
			P(\vx, \vy)=(\vx^\top\vx)(\vy^\top(I_n+B-A)\vy) + (\1_m^\top\vx)^2  \left(\vy^\top A \vy\right).
		\end{align*}
		Suppose that $P$ is PSD. The construction of its SOS expression proceeds in three steps.}
	
	\begin{enumerate}
		\item  \textbf{Form the matrix $M$.}
		Let $Q = I_n + B - A$ and $R = I_n + B + (m-1)A$.
		Then the $mn \times mn$ symmetric matrix $M$ is defined by
		\[
		M = I_m \otimes Q + \frac{1}{m}(\mathbf{1}_m\mathbf{1}_m^\top) \otimes (R - Q).
		\]
		
		\item \textbf{Compute a positive semidefinite factorization of $M$.}
		{It follows from the proof in Theorem~\ref{thm:sos_xsym} that} $M \succeq 0$. {Thus}, we can obtain vectors $\mathbf{w}_1,\dots,\mathbf{w}_r \in \R^{mn}$ such that
		\[
		M = \sum_{p=1}^{r} \mathbf{w}_p \mathbf{w}_p^\top .
		\]
		A convenient choice is the Cholesky decomposition $M = L L^\top$ (after a suitable permutation if $M$ is singular), in which case the $\mathbf{w}_p$'s are the columns of $L$.  Alternatively, one may use the spectral decomposition $M = U \Lambda U^\top$, set $\mathbf{w}_p = \sqrt{\lambda_p}\,\mathbf{u}_p$ (where $\lambda_p$ are the nonzero eigenvalues and $\mathbf{u}_p$ the corresponding eigenvectors), and then $r = \operatorname{rank}(M)$.
		
		\item \textbf{Extract the bilinear forms.}
		Each vector $\mathbf{w}_p$ can be reshaped into an $m \times n$ matrix $W_p$ by partitioning it into $m$ consecutive blocks of length $n$, i.e.,
		\[
		\mathbf{w}_p = \begin{pmatrix} \mathbf{w}_p^{(1)} \\ \vdots \\ \mathbf{w}_p^{(m)} \end{pmatrix},\qquad
		W_p = \bigl[ \mathbf{w}_p^{(1)} \;\cdots\; \mathbf{w}_p^{(m)} \bigr]^\top .
		\]
		Then the bilinear form $f_p$ is given by $f_p(\vx,\vy) = \vx^\top W_p \vy$, and we have
		\[
		P(\vx,\vy) = \sum_{p=1}^{r} f_p(\vx,\vy)^2 .
		\]
	\end{enumerate}
	
	If the original $x$-symmetric form is not monic, we first apply the reduction described in the proof of {Theorem}~\ref{cor:general_xsym}: for each $j$ with $d_j>0$ set $z_j = \sqrt{d_j}\,y_j$, thereby obtaining a monic form in the variables $\vx$ and $\vz$.  After constructing the SOS decomposition for the monic form, we substitute back $y_j = z_j/\sqrt{d_j}$ to obtain an SOS decomposition in the original variables.
	
	\medskip
	
	\noindent\textbf{Complexity and structure.}
	The matrix $M$ has size $mn$, which can be large when $m$ and $n$ are big.  However, its special Kronecker product structure can be exploited to accelerate the factorization.  Observe that $M$ is block circulant with respect to the $\vx$ indices.  By applying an orthogonal transformation that diagonalizes the matrix $\mathbf{1}_m\mathbf{1}_m^\top$, one can block  diagonalize $M$ into $m$ independent $n\times n$ matrices.  Specifically, let $U$ be an $m\times m$ orthogonal matrix whose first column is $\mathbf{1}_m/\sqrt{m}$ and whose remaining columns are any orthonormal basis of $\mathbf{1}_m^\perp$.  Then
	\[
	(U^\top \otimes I_n) \, M \, (U \otimes I_n) =
	\begin{bmatrix}
		R & & \\
		& Q & \\
		& & \ddots \\
		& & & Q
	\end{bmatrix},
	\]
	where the first block is {$R$,}
	and the remaining $m-1$ blocks are each $Q$.  Consequently, the eigenvalues of $M$ are precisely the eigenvalues of $R$ (with multiplicity $1$) together with the eigenvalues of $Q$ (with multiplicity $m-1$).  Moreover, the eigenvectors of $M$ can be obtained from those of $Q$ and $R$ via the same transformation.  	
	Using this structured approach, one may directly compute the eigenvalues and eigenvectors of $Q$ and $R$, then assemble the vectors $\mathbf{w}_p$ without ever forming the full matrix $M$.  We summarize the efficient procedure as follows.
	
	\begin{enumerate}
		\item Compute the spectral decompositions $Q = U_Q \Lambda_Q U_Q^\top$ and $R = U_R \Lambda_R U_R^\top$.
		\item For each positive eigenvalue $\lambda$ of $Q$ with eigenvector $\mathbf{u}$, produce $m-1$ vectors $\mathbf{w}$ of length $mn$ by setting, for $k=2,\dots,m$,
		\[
		\mathbf{w} = \mathbf{v}_k \otimes \sqrt{\lambda}\,\mathbf{u},
		\]
		where $\{\mathbf{v}_k\}_{k=2}^m$ is any orthonormal basis of $\mathbf{1}_m^\perp$ (e.g., the last $m-1$ columns of the $m\times m$ discrete cosine transform matrix).
		\item For each positive eigenvalue $\mu$ of $R$ with eigenvector $\mathbf{u}$, produce one vector
		\[
		\mathbf{w} = \frac{1}{\sqrt{m}}\mathbf{1}_m \otimes \sqrt{\mu}\,\mathbf{u}.
		\]
		\item Each such $\mathbf{w}$ gives a bilinear form as described above.
	\end{enumerate}
	
The use of spectral decomposition significantly reduces the computational cost of calculating the eigenpairs of $M$ from $O(m^3 n^3)$ to
$O(n^3 + m n^2)$, 	which is a substantial saving when $m$ is large.
Here,  the $O(n^3)$ term corresponds to the first step, while the second and third steps require  $O(m n^2)$ and  $O(n^2)$  operations, respectively.
	
	The number of squares obtained in this way equals $\operatorname{rank}(Q)\cdot(m-1) + \operatorname{rank}(R)$, which is at most $mn$ and often much smaller when $Q$ and $R$ are low rank.

	{\section{Extension to General Partially Symmetric Forms}}
	\label{subsec:extension_general_xsym}

	For general $x$-symmetric biquadratic forms,  the diagonal coefficients (coefficients of $x_i^2 y_j^2$) are independent of $i$; denote them by $d_j$ for $j=1,\dots,n$.
	Then,
	\begin{align}\label{general_P}
		P(\vx, \vy)=\sum_{i=1}^m \sum_{j=1}^n d_jx_i^2 y_j^2
		+ \sum_{i \neq k} \sum_{j, l=1}^n a_{jl} x_i x_k y_jy_l  + \sum_{i=1}^m \sum_{j \neq l} b_{jl}x_i^2 y_j y_l.
	\end{align}
	The SOS result for monic $x$-symmetric forms extends immediately to all $x$-symmetric PSD biquadratic forms, as the following {theorem} shows.
	
	\begin{Thm}\label{cor:general_xsym}
		Let $P$ be an $x$-symmetric biquadratic form (not necessarily monic). If $P$ is positive semi-definite, then $P$ is a sum of squares of bilinear forms.
	\end{Thm}
	
	\begin{proof}
		Evaluating $P$ {in \eqref{general_P}} at $\vx = \mathbf{e}_i$ and $\vy = \mathbf{e}_j$ gives $P(\mathbf{e}_i, \mathbf{e}_j) = d_j \ge 0$ for all $i,j$, because $P$ is PSD.
		Let $J = \{ j \in [n] : d_j > 0 \}$ and $J_0 = \{ j \in [n] : d_j = 0 \}$. We consider two cases.
		
		\textbf{Case 1: $J_0 = \emptyset$ (all $d_j > 0$).}
		Define a linear change of variables in $\vy$ by $z_j = \sqrt{d_j} y_j$ for $j=1,\dots,n$. Then
		\begin{align*}
			P(\vx, \vy) &= \tilde{P}(\vx, \vz)\\
			& := \sum_{j=1}^n \left( \sum_{i=1}^m x_i^2 \right) z_j^2  + \sum_{i \neq k} \sum_{j,l=1}^n \tilde{a}_{jl} x_i x_k z_j z_l + \sum_{i=1}^m \sum_{j,l=1}^n \tilde{b}_{jl} x_i^2 z_j z_l,
		\end{align*}
		where $\tilde{a}_{jl} = a_{jl}/\sqrt{d_j d_l}$ and $\tilde{b}_{jl} = b_{jl}/\sqrt{d_j d_l}$. Then $\tilde{P}$ is a monic $x$-symmetric biquadratic form (with all diagonal coefficients equal to $1$). Moreover, $\tilde{P}$ is PSD because $P$ is PSD and the transformation is invertible. By Theorem~\ref{thm:sos_xsym}, $\tilde{P}$ is SOS, i.e., there exist bilinear forms $f_p(\vx, \vz)$ such that $\tilde{P} = \sum_p f_p^2$. Substituting back $z_j = \sqrt{d_j} y_j$, each $f_p(\vx, \vz)$ becomes a bilinear form in $\vx$ and $\vy$ (since it is linear in $\vz$ and $\vz$ is linear in $\vy$). Hence $P$ is SOS.
		
		\textbf{Case 2: $J_0 \neq \emptyset$ (some $d_j = 0$).}
		For any $j_0 \in J_0$, we have $d_{j_0}=0$. Fix an arbitrary index $i_0 \in [m]$. Taking $\vx = \mathbf{e}_{i_0}$ and $\vy = \mathbf{e}_{j_0}$ yields $P(\mathbf{e}_{i_0}, \mathbf{e}_{j_0}) = d_{j_0} = 0$. Since $P$ is PSD, the quadratic form in $\vy$ given by $P(\mathbf{e}_{i_0}, \vy)$ is nonnegative and vanishes at $\vy = \mathbf{e}_{j_0}$. Therefore, the gradient with respect to $\vy$ at $\mathbf{e}_{j_0}$ must be zero,
		i.e.,
		\[\nabla_{\vy} P(\mathbf{e}_{i_0}, \mathbf{e}_{j_0})=2\begin{bmatrix}
			b_{1j_0}&\cdots & b_{{j_0-1}j_0} & 0 & b_{(j_0+1)j_0}&\cdots & b_{{n}j_0}	
		\end{bmatrix}^\top=0.\]
		Hence, we get $b_{j j_0} = 0$ for all $j$.
		
		Similarly, 	 {the quadratic form in $\vx$ given by $P(\vx,\mathbf{e}_{j_0})$  is nonnegative and vanishes at $\vx = \mathbf{e}_{i_0}$. Thus,}
		it follows from \[\nabla_{\vx} P(\mathbf{e}_{i_0}, \mathbf{e}_{j_0})=2a_{j_0,j_0}(\1_m-\mathbf{e}_{i_0})=0 \]
		that $a_{j_0,j_0}=0$.
		Furthermore, {the quadratic form in  $\vy$ given by $P(\frac{1}{\sqrt{m}}\1_m,\vy)$  is  nonnegative and vanishes at $\vy = \mathbf{e}_{j_0}$. Thus,}
		\[\nabla_{\vy} P(\frac{1}{\sqrt{m}}\1_m, \mathbf{e}_{j_0})=2(m-1)\begin{bmatrix}
			a_{1j_0}&\cdots & a_{{j_0-1}j_0} & 0 & a_{(j_0+1)j_0}&\cdots & a_{{n}j_0}	
		\end{bmatrix}^\top\]
		implies $a_{j j_0} = 0$ for all $j$. % <- changed "derives" to "implies"
		Consequently,  the variable $y_{j_0}$ does not appear in any term of $P$.
		Therefore, $P$ does not depend on $y_{j_0}$ for any $j_0 \in J_0$.
		
		Thus, $P$ can be viewed as an $x$-symmetric biquadratic form in the variables $\vx$ and $\vy' = (y_j)_{j \in J}$, and for $j \in J$ we have $d_j > 0$. This reduces to Case~1. (If $J$ is empty, then $P$ does not depend on $\vy$ at all and is identically zero, which is trivially SOS.)
	\end{proof}
	
	\begin{remark}
		{Theorem}~\ref{cor:general_xsym} completely resolves the question for the $x$-symmetric class:
		every PSD $x$-symmetric biquadratic form, regardless of its diagonal values, admits an SOS decomposition.
	\end{remark}

	\begin{remark}\label{rem:nonmonic_rank_bound}
		When the original $x$-symmetric form is not monic, the scaling argument of {Theorem}~\ref{cor:general_xsym} reduces the problem to the monic case, possibly with a smaller number $n'$ of active $y$-variables.
		Applying Theorem~\ref{thm:sos_rank_bound} to the reduced form yields the same type of bound, now involving the ranks of the reduced matrices $\widetilde{Q}$ and $\widetilde{R}$.
	\end{remark}

	\begin{remark}
		By swapping the roles of $\vx$ and $\vy$, the same result holds for $y$-symmetric biquadratic forms: every PSD $y$-symmetric biquadratic form is SOS.   Thus, every PSD partially symmetric biquadratic form is SOS.
	\end{remark}

{\section{Lower Bounds for SOS Rank in Small Dimensions}}
	
	%\section{Tightness of the SOS Rank Bounds for  $2 \times 2$ and $3 \times 2$ Forms} % <- corrected "Tigntnesss" to "Tightness"
In this section, we use simple biquadratic forms as a tool to present some lower bounds for SOS rank in small dimensions. 

\medskip

\subsection{Simple Biquadratic Forms}

Let \(m \ge n\).   We say a biquadratic form is a \textbf{simple biquadratic forms} if it contains  only distinct terms of the type {\(x_i^2 y_j^2\).}   Then we define a simple biquadratic form series \(P_{m,n,s}\) as follows.
For each \(s = 1,2,\dots,mn\), the form \(P_{m,n,s}\) contains exactly \(s\) distinct terms of the type \(x_i^2 y_j^2\).

{Assume  that} the terms are taken in the following order.
For \(k = 0,1,\dots,mn-1\), write \(k = p m + q\) with integers {\(0\le p <n\)} and \(0 \le q < m\)
(so \(p = \lfloor k/m \rfloor\) and \(q = k \bmod m\)).
Define
\[
i_k = q+1, \qquad j_k = {\big( p + q\big)} \bmod n  + 1.
\]
Then the sequence of index pairs is \((i_0,j_0), (i_1,j_1), \dots, (i_{mn-1},j_{mn-1})\).
We set
\[
P_{m,n,s}(\vx,\vy) = \sum_{k=0}^{s-1} x_{i_k}^2 y_{j_k}^2.
\]

Thus \(P_{m,n,s}\) has exactly \(s\) distinct square terms, and \(P_{m,n,mn}\) contains every term
\(x_i^2y_j^2\) (\(i=1,\dots,m\), \(j=1,\dots,n\)) exactly once.

For the small dimensions used in our theorems, the resulting forms are:

\noindent\textbf{Case \(m=2, n=2\):}
\[
\begin{aligned}
P_{2,2,1} &= x_1^2 y_1^2, \\
P_{2,2,2} &= x_1^2 y_1^2 + x_2^2 y_2^2, \\
P_{2,2,3} &= x_1^2 y_1^2 + x_2^2 y_2^2 + x_1^2 y_2^2, \\
P_{2,2,4} &= x_1^2 y_1^2 + x_2^2 y_2^2 + x_1^2 y_2^2 + x_2^2 y_1^2 .
\end{aligned}
\]

\noindent\textbf{Case \(m=3, n=2\):}
\[
\begin{aligned}
P_{3,2,1} &= x_1^2 y_1^2, \\
P_{3,2,2} &= x_1^2 y_1^2 + x_2^2 y_2^2, \\
P_{3,2,3} &= x_1^2 y_1^2 + x_2^2 y_2^2 + x_3^2 y_1^2, \\
P_{3,2,4} &= x_1^2 y_1^2 + x_2^2 y_2^2 +  x_3^2 y_1^2 + x_1^2 y_2^2 , \\
P_{3,2,5} &= x_1^2 y_1^2 + x_2^2 y_2^2 + x_3^2 y_1^2 + x_1^2 y_2^2  + x_2^2 y_1^2, \\
P_{3,2,6} &= x_1^2 y_1^2 + x_2^2 y_2^2 + x_3^2 y_1^2 + x_1^2 y_2^2 + x_2^2 y_1^2 + x_3^2 y_2^2 .
\end{aligned}
\]

\noindent\textbf{Case \(m=3, n=3\):}
\[
\begin{aligned}
P_{3,3,1} &= x_1^2 y_1^2, \\
P_{3,3,2} &= x_1^2 y_1^2 + x_2^2 y_2^2, \\
P_{3,3,3} &= x_1^2 y_1^2 + x_2^2 y_2^2 + x_3^2 y_3^2, \\
P_{3,3,4} &= x_1^2 y_1^2 + x_2^2 y_2^2 + x_3^2 y_3^2 + x_1^2 y_2^2, \\
P_{3,3,5} &= x_1^2 y_1^2 + x_2^2 y_2^2 + x_3^2 y_3^2 + x_1^2 y_2^2 + x_2^2 y_3^2, \\
P_{3,3,6} &= x_1^2 y_1^2 + x_2^2 y_2^2 + x_3^2 y_3^2 + x_1^2 y_2^2 + x_2^2 y_3^2 + x_3^2 y_1^2 .
\end{aligned}
\]

\medskip

\subsection{Tightness Examples and a General Lower Bound}

Theorem 1 of \cite{CQX25} shows that a $2 \times 2$ PSD biquadratic form can always be expressed as the sum of three squares.
 Theorem 2 of \cite{QCX25} shows that a $3 \times 2$ PSD biquadratic form can always be expressed as the sum of four squares.
The following theorem shows that these two bounds are tight, and present a general lower bound for $m \ge n=2$.

\begin{Thm}[A general lower bound for $m \times 2$ forms]\label{thm:m2-lower}
Let $m \ge 2$.  The simple biquadratic form
\[
P_{m,2,m+1}(\vx,\vy) = \sum_{k=0}^{m} x_{i_k}^2 y_{j_k}^2,
\]
where $(i_k,j_k)$ are defined by the ordering rule in Section~2, is positive semidefinite and satisfies
\[
\operatorname{SOS\text{-}rank}(P_{m,2,m+1}) = m+1.
\]
In particular, $P_{m,2,m+1}$ cannot be written as a sum of $m$ squares of bilinear forms.
\end{Thm}

\begin{proof}
The nonnegativity and the upper bound $\operatorname{SOS\text{-}rank}(P_{m,2,m+1}) \le m+1$ are immediate from the decomposition
\[
P_{m,2,m+1} = \sum_{k=0}^{m} (x_{i_k} y_{j_k})^2 .
\]

For the lower bound, assume for contradiction that $P_{m,2,m+1} = \sum_{t=1}^{m} L_t^2$ with each $L_t$ bilinear.
Write $L_t = \sum_{i=1}^{m} \sum_{j=1}^2 c_{ij}^{(t)} x_i y_j$.
Let $A_i = (c_{i1}^{(t)})_{t=1}^m \in \R^m$ and $B_i = (c_{i2}^{(t)})_{t=1}^m \in \R^m$.

From the square terms present in $P_{m,2,m+1}$ one obtains the following conditions:
\begin{itemize}
\item For the index $i_0=1$, both $x_1^2y_1^2$ and $x_1^2y_2^2$ appear; hence $\|A_1\|^2 = 1$ and $\|B_1\|^2 = 1$.
\item For each $i = 2,\dots,m$, exactly one of $x_i^2y_1^2$ or $x_i^2y_2^2$ appears, depending on the parity of $i$ in the ordering.  Consequently, for each such $i$, either $\|A_i\|^2=1$ and $B_i = 0$, or $\|B_i\|^2=1$ and $A_i = 0$.
\end{itemize}
The form $P_{m,2,m+1}$ contains no mixed terms $x_p x_q y_r y_s$ with $p \neq q$ or $r \neq s$.  Setting the coefficients of all such mixed monomials in $\sum_t L_t^2$ to zero yields the orthogonality relations
\[
A_p \cdot A_q = 0 \;(p\neq q), \qquad B_p \cdot B_q = 0 \;(p\neq q), \qquad A_p \cdot B_q = 0 \;(\text{all }p,q).
\]
Thus the nonzero vectors among $\{A_1,B_1,A_2,B_2,\dots,A_m,B_m\}$ are pairwise orthogonal.  Because of the square?term conditions, exactly $m+1$ of these vectors are nonzero and have norm $1$ (namely $A_1$, $B_1$, and one vector for each $i=2,\dots,m$).  We therefore have $m+1$ mutually orthogonal unit vectors in $\R^m$, which is impossible.  Hence no $m$-square decomposition exists, and $\operatorname{SOS\text{-}rank}(P_{m,2,m+1}) \ge m+1$.
\end{proof}

{\begin{remark} 
Fang and Huang \cite{FH26} informed us that they obtained the same lower bound as the above theorem but with a different method.
\end{remark}}

\medskip

\subsection{A Lower Bound for SOS Rank of $3 \times 3$ Biquadratic Forms}

Denote the SOS rank of a bilinear form by $\operatorname{sos}(P)$.

{\begin{Thm}[A $3 \times 3$ form requiring six squares]\label{thm:3x3-tight}
    Let
    \[
    P'(\vx,\vy) = P_{3,3,6}(\vx,\vy) = x_1^2 y_1^2 + x_2^2 y_2^2 + x_3^2 y_3^2 + x_1^2 y_2^2 + x_2^2 y_3^2 + x_3^2 y_1^2,
    \]
    where $\vx = (x_1,x_2,x_3),\vy = (y_1,y_2,y_3) \in \R^3$.
    Then
    \[
    \operatorname{sos}(P') = 6.
    \]
    %Consequently,
    %\[
    %\mathrm{BRS}(3,3) \ge 6.
    %\]
\end{Thm}

\begin{proof}
    The upper bound $\operatorname{sos}(P') \le 6$ is immediate by taking
    \[
    L_1 = x_1y_1,\; L_2 = x_2y_2,\; L_3 = x_3y_3,\;
    L_4 = x_1y_2,\; L_5 = x_2y_3,\; L_6 = x_3y_1,
    \]
    which yields $P' = \sum_{i=1}^6 L_i^2$.

    For the lower bound, assume for contradiction that $P' = \sum_{i=1}^5 L_i^2$ with each $L_i$ bilinear.
    Write $L_i = \sum_{a,b=1}^3 c_{ab}^{(i)} x_a y_b$.
    Because $P'$ contains exactly the six square terms $x_a^2 y_b^2$ for\\
    $(a,b) \in S = \{(1,1),(2,2),(3,3),(1,2),(2,3),(3,1)\}$ with coefficient $1$ and no other square terms,
    we must have
    \[
    \sum_{i=1}^5 (c_{ab}^{(i)})^2 = 1 \quad ((a,b)\in S), \qquad
    \sum_{i=1}^5 (c_{ab}^{(i)})^2 = 0 \quad ((a,b)\notin S).
    \]
    Hence each $L_i$ involves only the six coefficient types corresponding to $S$:
    \[
    L_i = a_i x_1y_1 + b_i x_1y_2 + c_i x_2y_2 + d_i x_2y_3 + e_i x_3y_3 + f_i x_3y_1.
    \]
    Define vectors $A,B,C,D,E,F \in \R^5$ by $A=(a_1,\dots,a_5)$, etc.
    From the square-term conditions we have $\|A\|=\|B\|=\cdots=\|F\|=1$.

    Since $P'$ contains no mixed terms $x_px_qy_ry_s$ with $p\neq q$ or $r\neq s$,
    all cross terms in $\sum_{i=1}^5 L_i^2$ must vanish.
    This yields the orthogonality conditions
    \[
    A\!\cdot\!B = A\!\cdot\!C = \dots = E\!\cdot\!F = 0,
    \]
    i.e., the six vectors $A,\dots,F$ in $\R^5$ are pairwise orthogonal and each has norm $1$.
    This is impossible because $\R^5$ can contain at most five nonzero mutually orthogonal vectors.
    Therefore no such decomposition with five squares exists, and $\operatorname{sos}(P') \ge 6$.
\end{proof}}
	
 \section{An {Improved} Universal SOS Rank Bound for General Biquadratic Forms}
\label{sec:universal_bound}

The next theorem presents an explicit expression for the SOS Rank of a biquadratic form. Firstly,  recall that \cite{CQX25} any $m\times n$ biquadratic polynomial can be written as follows
	\begin{equation}\label{equ:M_gamma}
		P(\vx, \vy) = \sum_{i,k=1}^m \sum_{j,l=1}^n a_{ijkl}x_i y_j x_k y_l=\vz^\top (B+P(\Gamma))\vz:=\vz^\top M(\Gamma)\vz,
	\end{equation}
	where $\vz=\vx\otimes \vy$, $B_{(ij)(kl) = a_{ijkl}}\in\mathbb R^{mn\times mn}$, $\Gamma \in \mathbb R^{\binom{m}2 \times \binom{n}2}$ is a symmetric parameter matrix that captures the degrees of freedom in expressing the full cross terms with $i\neq k$ and $j\neq l$, and $P(\Gamma)$ is detailed in \cite{CQX25}.
	\begin{Thm} \label{t6.1}
		Suppose that the $m\times n$ biquadratic form in \eqref{equ:M_gamma} is SOS. Then we have
			\begin{equation}\label{def:SOS_rank}
		\operatorname{SOS\text{-}rank}(P)=\min_{\Gamma: M(\Gamma)\succeq 0}\;\operatorname{rank}(M(\Gamma)).
		\end{equation}
	\end{Thm}
	\begin{proof}
		Suppose that $P$ admits an SOS decomposition  $\sum_{r=1}^R f_r^2$.  Then, it holds that $\text{supp}(f_r)\subseteq \frac12 \text{New}(P)$ \cite[Theorem~1]{Re78}\cite[Theorem~4]{CQX25}. Here, $\text{New}(P)$ is the Newton polytope of $P$. From this  we may derive that the SOS decomposition can only be the SOS of bilinear forms, i.e., $\sum_{r} (\vc_r^\top\vz)^2$. Thus,	the SOS rank of $P$ equals the minimum possible rank of $M(\Gamma)$ satisfying  $M(\Gamma)\succeq0$. 		
	\end{proof}

While the previous sections focused on $x$-symmetric forms, the tool provided by Theorem~6.1 allows us to derive a universal upper bound on the SOS rank that holds for \emph{any} SOS biquadratic form, regardless of symmetry.

\begin{Thm}\label{thm:universal_bound}
For any integers $m,n\ge 2$, every SOS biquadratic form $P\in\R[x_1,\dots,x_m,y_1,\dots,y_n]$ satisfies
\[
\operatorname{SOS\text{-}rank}(P)\le mn-1.\]
\end{Thm}

\begin{proof}
Let $P$ be SOS. By Theorem~6.1,
\[
\operatorname{SOS\text{-}rank}(P)=\min_{\Gamma:M(\Gamma)\succeq 0}\rank(M(\Gamma)),
\]
where $M(\Gamma)$ is an $mn\times mn$ symmetric matrix representing $P$ as $P(\mathbf{x},\mathbf{y})=\mathbf{z}^\top M(\Gamma)\mathbf{z}$ with $\mathbf{z}=\mathbf{x}\otimes\mathbf{y}$.
The set $\mathcal{M}_P=\{M(\Gamma):\Gamma\}$ is an affine subspace of symmetric matrices of dimension $d=\binom{m}{2}\binom{n}{2}\ge 1$ (see \cite{CQX25}).
Let $\mathcal{M}_P^+=\{M\in\mathcal{M}_P:M\succeq 0\}$; since $P$ is SOS, $\mathcal{M}_P^+\neq\emptyset$.

Take any $M_0\in\mathcal{M}_P^+$. If $\rank(M_0)\le mn-1$, we are done.
If $\rank(M_0)=mn$, then $M_0$ is positive definite.
Because $\dim\mathcal{M}_P\ge1$, we can choose a nonzero direction $\Delta\in\mathcal{M}_P-M_0$ and consider the line $M(t)=M_0+t\Delta$.
Since the PSD cone is pointed, there exists $t^*>0$ such that $M(t)\succeq 0$ for $t\in[0,t^*]$ and $M(t^*)$ lies on the boundary of the PSD cone, i.e., $\rank(M(t^*))\le mn-1$.
Thus $M(t^*)\in\mathcal{M}_P^+$ and has rank at most $mn-1$.
Hence the minimum rank over $\mathcal{M}_P^+$ is at most $mn-1$, and by Theorem~5.1 this minimum equals $\operatorname{SOS\text{-}rank}(P)$.
\end{proof}

%\begin{remark}
%This bound is sharp for many small cases: for $2\times2$ forms the maximal SOS rank is $3=4-1$ (Theorem~5.1). {However,} for $3\times2$ forms the maximal SOS rank is $4=6-2$ (Theorem~5.2). Whether the bound $mn-1$ is always attainable for $m,n\ge3$ remains an open question.
%\end{remark}

\begin{remark}
The bound $mn-1$ improves the trivial bound $mn$ obtained by the vectorization $\mathbf{z}=\mathbf{x}\otimes\mathbf{y}$.
For structured classes such as $x$-symmetric forms, tighter bounds are given by Theorem~\ref{thm:sos_rank_bound}; for example, when $Q$ and $R$ are low-rank, the SOS rank can be much smaller than $mn-1$.
\end{remark}

{

To quantify the worst-case SOS rank across all biquadratic forms, we introduce the following quantity:

\begin{Def}
Let $m, n \ge 2$.  Let $BSR(m, n)$ be the maximum sos rank of $m \times n$ SOS biquadratic forms.
\end{Def}

Then from the discussion in the last section and this section, we have the following theorem.

\begin{Thm}

We have the following conclusions.

1. $BSR(m, n) = BSR(n, m) \le mn-1$ for all $m, n \ge 2$.

2.  $BSR(2, 2) = 3$.

3. $BSR(3, 2) = 4$.

4. {$BSR(m_1, n_1) \le BSR(m_2, n_2)$ if $m_1 \le m_2$ and $n_1 \le n_2$.}

\end{Thm}
}

{ \begin{Cor} \label{cor:BRS33-updated}
    For $m=n=3$, we have
    \[
    6 \; \le \; \mathrm{BSR}(3,3) \; \le \; 8,
    \]
    where the lower bound follows from Theorem~\ref{thm:3x3-tight}.
\end{Cor}

\begin{remark}
    Determining the exact value of $\operatorname{BRS}(3,3)$¡ªwhether it equals $6$, $7$ or $8$, remains an open problem.
\end{remark}}
	
	% ========== SECTION 5: CONCLUDING REMARKS ==========
	\section{Concluding Remarks}
	\label{sec:conclusion}
	
	In this paper, we have studied the sum-of-squares property for partially symmetric biquadratic forms. Our main results can be summarized as follows:
	
	\begin{itemize}
		\item We established necessary and sufficient conditions for positive semi-definiteness of monic $x$-symmetric biquadratic forms (Theorem~\ref{thm:psd_xsym}).
		\item We proved that every PSD $x$-symmetric biquadratic form is SOS (Theorem~\ref{thm:sos_xsym}), extending the known result for fully symmetric forms \cite{XCQ25} to the partially symmetric setting.
		\item We derived an explicit upper bound on the SOS rank for such forms, expressed in terms of the ranks of two associated matrices (Theorem~\ref{thm:sos_rank_bound}).
		\item We provided an efficient computational procedure for constructing SOS decompositions, exploiting the Kronecker-product structure to reduce the cost from $O(m^3n^3)$ to $O(n^3+mn^2)$ (Section~\ref{sec:computation}).
		\item We demonstrated the tightness of the known SOS rank bounds for $2 \times 2$ and $3 \times 2$ biquadratic forms by exhibiting {an $m \times 2$ form that requires exactly $m+1$ squares (Theorem~\ref{thm:m2-lower}).  This also establishes a general lower bound $m+1$ for the sos rank of a $m \times 2$ PSD biquadratic form.}
        \item {We present an $3 \times 3$ SOS biquadratic form, which can be expressed as the sum of six squares, but not the sum  of five squares (Theorem~\ref{thm:3x3-tight}).}
         \item We proved {an improved} universal upper bound $\operatorname{SOS\text{-}rank}(P) \le mn-1$ for any SOS biquadratic form.
	\end{itemize}
	
	These results completely resolve the SOS question for the class of partially symmetric biquadratic forms, showing that partial symmetry either in the $\vx$ or $\vy$ variables is sufficient to guarantee an SOS decomposition for all PSD forms in the class. % <- corrected "symmetr" to "symmetry"
	
	\medskip

    \noindent \textbf{Open problems and future work.}
Several natural questions remain:
\begin{enumerate}
    \item {What is the exact value of \(\mathrm{BSR}(3,3)\)? Theorem~\ref{thm:3x3-tight} shows \(\mathrm{BSR}(3,3)\ge 6\), while Corollary~\ref{cor:BRS33-updated} gives the upper bound \(8\). Closing the gap \(6\le\mathrm{BSR}(3,3)\le 8\) is a natural next step.}

    \item What is the maximal possible SOS rank for an \(m \times n\) PSD \(x\)-symmetric biquadratic form?
    Theorem~\ref{thm:sos_rank_bound} gives an upper bound \(\operatorname{rank}(R)+(m-1)\operatorname{rank}(Q)\), which can be as large as \(mn\).
    However, the universal bound in Theorem~\ref{thm:universal_bound} shows that \(\operatorname{SOS\text{-}rank}(P)\le mn-1\) for any SOS biquadratic form.
    It remains open whether \(mn-1\) can be attained by an \(x\)-symmetric form, and how the worst-case growth with \(m\) and \(n\) behaves under the \(x\)-symmetry constraint.

    \item Can similar SOS guarantees be established for biquadratic forms with other types of symmetry, such as block symmetry or cyclic symmetry?
    \item How do these results extend to higher-degree multiquadratic forms (e.g., trilinear forms raised to the fourth power)?
    \item Are there practical applications in optimization or engineering where the efficient SOS construction of Section~\ref{sec:computation} can be deployed at scale?
\end{enumerate}

	We hope that the techniques developed here especially the use of the Kronecker-product representation and the block-diagonalization trick will be useful in tackling these and related questions in the future. % <- added dash for clarity
	
	\bigskip
	
	\noindent\textbf{Acknowledgement}
	This work was partially supported by Research Center for Intelligent Operations Research, The Hong Kong Polytechnic University (4-ZZT8), the National Natural Science Foundation of China (Nos. 12471282 and 12131004),  and Jiangsu Provincial Scientific Research Center of Applied Mathematics (Grant No. BK20233002).
	
	\medskip
	
	\noindent\textbf{Data availability}
	No datasets were generated or analysed during the current study.
	
	\medskip
	
	\noindent\textbf{Conflict of interest} The authors declare no conflict of interest.

\end{document}